\documentclass[a4paper]{amsart}

\usepackage{amsthm,amsfonts,amsmath, amssymb}

\let\oldmarginpar\marginpar
\renewcommand{\marginpar}[1]{\oldmarginpar{\small\textit{{#1}}}}

\usepackage{setspace}
\usepackage{graphicx,float,enumerate,subfigure,color}
\usepackage[ruled,boxed,commentsnumbered,norelsize]{algorithm2e}
\usepackage{verbatim}
\usepackage{url}
\usepackage[capitalise]{cleveref}

\usepackage[sort&compress,numbers]{natbib}
\usepackage{enumitem}
\setlist[enumerate,1]{label=(\roman*),font=\normalfont}

\newtheorem{theorem}{Theorem}
\newtheorem{corollary}[theorem]{Corollary}
\newtheorem{lemma}[theorem]{Lemma}
\newtheorem{proposition}[theorem]{Proposition}

\theoremstyle{definition}

\theoremstyle{remark}
\newtheorem{remark}[theorem]{Remark}

\crefname{remark}{Remark}{Remarks}

\theoremstyle{definition}

\usepackage{etoolbox}

\theoremstyle{remark}

\usepackage{etoolbox}

\crefname{rmk}{Remark}{Remarks}
\crefname{problem}{Problem}{Problems}

\usepackage{mathtools}

\DeclareMathOperator{\conv}{conv}

\newcommand{\R}{\mathbb{R}}
\newcommand{\h}{\frac{1}{2}}

\renewcommand{\ge}{\geqslant}
\renewcommand{\le}{\leqslant}

\date{\today}
\title{A lower bound theorem for $d$-polytopes with $2d+1$ vertices}

\usepackage{amsaddr}

\author{Guillermo Pineda-Villavicencio}
\address{School of Information Technology, Deakin University, Geelong,  Australia\\Federation University, Ballarat,  Australia}
\email{\texttt{work@guillermo.com.au}}

\author{David Yost}
\address{Federation University, Ballarat, Australia}
\email{\texttt{d.yost@federation.edu.au}}


\begin{document}
\begin{abstract} The problem of calculating exact lower bounds for the number of $k$-faces of $d$-polytopes with $n$ vertices, for each value of $k$, and characterising the minimisers, has recently been solved for $n\le2d$. We establish the corresponding result for $n=2d+1$; the nature of the lower bounds and the minimising polytopes are quite different in this case. As a byproduct, we also characterise all $d$-polytopes with $d+3$ vertices, and only one or two edges more than the minimum.
\end{abstract}

\maketitle
\section{Introduction}

Maximisation and minimisation problems are ubiquitous in mathematics. Here we focus on the problem of minimising an important invariant, namely the number of faces of a given dimension, over the class of polytopes of a given dimension with a given number of vertices. Such lower bound problems have attracted the attention of numerous researchers over the years \cite{KleNevNov19,Bar73, McM71,Xue21}.

A polytope $P$ of dimension $d$ is a \textit{$d$-polytope} and a $k$-dimensional face of the polytope is a $k$-face. A \textit{facet} is a $(d-1)$-face, a \textit{ridge} is a $(d-2)$-face, an \textit{edge} is a 1-face, and a vertex is a 0-face.  The number of $k$-faces of  $P$ is denoted by $f_k(P)$.

The \textit{graph} of a polytope is the graph formed by the vertices and edges of the polytope. A vertex in a $d$-polytope is \textit{simple} if it is contained in  exactly $d$ edges; otherwise it is \textit{nonsimple}. And a polytope is \textit{simple} if all its vertices are simple; otherwise it is \textit{nonsimple}. A polytope is \textit{simplicial} if all its facets are simplices.

We require some familiarity with duality of polytopes: two polytopes $P$ and $Q$ of the same dimension are called {\it dual} to one another if there is an inclusion reversing bijection between the set of all faces of $P$ and the set of all faces of $Q$. This implies that vertices of $P$ correspond to facets of $Q$, edges of $P$ correspond to ridges of $Q$, and so on. The existence of duals is assured by the \textit{polar} construction: assuming (without loss of generality) that the interior of $P$ contains the origin, $P^*$ can be defined as
 \[P^*=\left\{y\in \R^d\middle|\; x\cdot y\le1\; \text{for all $x$ in $P$}\right\}, \]
 where $\cdot$ denotes as usual the dot product of two vectors. It is routine to show that a polytope is simplicial if and only if any polytope dual to it is simple. For further details see \cite[Sec. 3.4]{Gru03} or \cite[Sec. 2.3]{Zie95}.

The earliest and best known lower bound theorem is due to Barnette \cite{Bar71,Bar73}. We state it in its original formulation for simple polytopes.

\begin{proposition}[Simple polytopes, {\cite{Bar71,Bar73}}]\label{prop:simple} Let $d\ge 2$ and let $P$ be a simple $d$-polytope with $f_{d-1}$ facets. Then
\begin{equation}\label{eq:simple}
 f_k(P)\ge\begin{cases}
(d-1)f_{d-1}-(d+1)(d-2),& \text{if $k=0$}; \\
\binom{d}{k+1}f_{d-1}-\binom{d+1}{k+1}(d-1-k),& \text{if $k\in [1,d-2]$}.
\end{cases}
\end{equation}
Moreover, for every value of $f_{d-1}\ge d+1$, there are simple $d$-polytopes for which equality holds.
\end{proposition}

Dually, Barnette's Theorem also gives exact lower bounds for the numbers of faces of each positive dimension for the class of simplicial polytopes, with any fixed number of vertices. However the corresponding problem for {\it general} polytopes has not received so much attention. It is necessary to break this problem down according to the number of vertices.

Gr\"unbaum \cite[Sec.~10.2]{Gru03} defined a function $\phi_k(d+s,d)$, for  $s\le d$, by:
\begin{equation}\label{eq:at-most-2d}
   \phi_k(d+s,d)=\binom{d+1}{k+1}+\binom{d}{k+1}-\binom{d+1-s}{k+1}.
\end{equation}
He conjectured that $\phi_k(d+s,d)$ gives the minimum number of $k$-faces of a $d$-polytope with $d+s$ vertices   \cite[Sec.~10.2]{Gru03}; this was recently proved in full generality by Xue \cite{Xue21}, who also characterised the unique minimisers for $k\in[1,d-2]$; see \cref{prop:at-most-2d}. This had been proved earlier for some restricted values of $k$ in \cite{PinUgoYosLBT}.

The minimisers for $\phi_k(d+s,d)$ are called triplices and were introduced in \cite[Sec.~3]{PinUgoYosLBT}.
The \textit{($s$,$d-s$)-triplex} $M(s,d-s)$ is defined as a  $(d-s)$-fold pyramid over a  simplicial $s$-prism for $s\in [1, d]$. In particular,  $M(1,d-1)$ is a $d$-simplex $\Delta(d)$ and $M(d,0)$ is a simplicial $d$-prism. The \textit{prism} over a polytope $Q$  is the product of $Q$ and a line segment, or any polytope combinatorially equivalent to it. By a \textit{simplicial $d$-prism} we mean a prism  over a $(d-1)$-simplex.

A simplicial $d$-prism has $d+2$ facets: two $(d-1)$-simplices and $d$ simplicial $(d-1)$-prisms. Its edges fall naturally into three types: edges of one simplex facet, edges of the other simplex facet, and edges with one vertex in each simplex facet. We remark that edges of the latter type cannot be skew. Indeed, any two such edges must determine a 2-face, and hence they are coplanar. It follows that all $d$ such edges are either parallel or contained in concurrent lines.

\begin{proposition}[$d$-polytopes with at most $2d$ vertices, \cite{Xue21}]\label{prop:at-most-2d} Let $d\ge 2$ and let $P$ be a  $d$-polytope with  $d+s$ vertices, where $s\le d$. Then
\[f_k(P)\ge \phi_k(d+s,d)\quad \text{for all $k\in[1,d-1]$}.\]
Also, if $f_k(P)=\phi_k(d+s,d)$ for some $k\in[1,d-2]$, then $P$ is the $(s,d-s)$-triplex.
\end{proposition}

Note that the case $k=d-1$ is not included in the last part of \cref{prop:at-most-2d}, as it behaves a bit differently. The triplex $M(s,d-s)$ has exactly $\phi_{d-1}(d+s,d)=d+2$ facets, but it may not be the only minimiser. For example, a pyramid over $\Delta(2,2)$ (defined below) is a 5-polytope with ten vertices and seven facets, but it is not a triplex.

The version of \cref{prop:at-most-2d}  for facets (i.e. the case $k=d-1$) was investigated thoroughly by McMullen \cite{McM71}, who considered $d$-polytopes with even more than $2d$ vertices. He showed that the minimal number of facets, for polytopes satisfying $2d\le f_0\le\frac{1}{4}d^2+2d$,   is either $d+2$ or $d+3$, depending on number theoretic properties of $d$ and $f_0$. For the case of $2d+1$ vertices, we have $\psi_{d-1}(d)=d+3$ when $d$ is prime, while $\psi_{d-1}(d)=d+2$ when $d$ is composite. Our main result, \cref{thm:2dplus1-bound}, shows that this sort of dichotomy is quite common.

Hitherto, apart from the work of McMullen, there has not even been a conjecture for $d$-polytopes with more than $2d$ vertices. It may have seemed that such polytopes were too chaotic to have much structure. We
show that this is not so
by solving the problem under discussion for $d$-polytopes with $2d+1$ vertices. That is, we establish the minimum number $\psi_k(d)$ of $k$-faces of a $d$-polytope with $2d+1$ vertices, and  characterise the minimisers (\cref{thm:2dplus1-bound}). We require different techniques for this new case, and the nature of the minimising polytopes is also quite different. (For more vertices, the technical difficulties increase. However, we have also found the minimal number of edges of $d$-polytopes with $2d+2$ vertices, and characterised the minimisers \cite{PinUgoYos2d+2}.)

With one exception in dimension three, namely the polytope $\Sigma(3)$,  each minimiser for  $\psi_k(d)$ is either a $d$-pentasm or a certain $d$-polytope with $d+2$ facets. Accordingly, we need to describe these classes of examples. We do this briefly here, with more details in \cref{subsection:extremalexamples} below.

The \textit{pentasm} $Pm(d)$ in dimension $d$, or \textit{$d$-pentasm}, was defined in \cite[Sec. 4]{PinUgoYosLBT}; it can be obtained by  truncating a simple vertex from the triplex $M(2,d-2)$, i.e. by intersecting this triplex with a closed halfspace which contains all the nonsimple vertices and three of the four simple vertices. For more details about truncation, see \cref{subsection:trunc} below.

The {\it(Minkowski) sum} of two polytopes $Q+R$ is defined to be $\{x+y: x\in Q, y\in R\}$. The simple polytope $\Delta({r,s})$ (with $r,s>0$) is defined as the sum of an $r$-dimensional simplex and an $s$-dimensional simplex, lying in complementary subspaces; it has $r+s+2$ facets.  It turns out that $t$-fold pyramids over such simple polytopes, which we denote by $\Delta^{t}(r,s)$,  are the only $d$-polytopes with $d+2$ facets \cite{McM71}; see \cref{lem:dplus2facets} below.

Finally, in  $\mathbb{R}^3$, we denote by $\Sigma(3)$ any polytope that is combinatorially equivalent to the convex hull of
 $\left\{{0},{e_1},{e_1}+{e_3},{e_2},{e_2}+{e_3},{e_1}+{e_2},{e_1}+{e_2}+2{e_3}\right\}$,
where ${e_i}$ is the standard $i^{th}$  unit vector. Like the 3-pentasm, this polytope has seven vertices and eleven edges. (There is a higher dimensional version of this polytope, but it has $3d-2$ vertices \cite[p. 2016]{PinUgoYosEXC} and so is only of interest in this paper when $d=3$.)

Examples of all the aforementioned polytopes are depicted in \cref{fig:2dplus1}. The next theorem summarises our results; the \textit{$f$-vector} of a $d$-polytope $P$, denoted $f(P)$, is the sequence $(f_{0}(P),\ldots,f_{d-1}(P))$ of the numbers of faces of $P$ of different dimensions.

\begin{theorem}
\label{thm:2dplus1-bound} Let $d\ge3$ and consider the class of $d$-polytopes with $2d+1$ vertices. Fix $k\in[1,d-2]$. Then the following hold.
\begin{enumerate}
    \item Let $d=3$. If $P$ is $\Sigma(3)$ or a pentasm, then $f(P)=(7,11,6)$. Otherwise, $f_1(P)>11$ and $f_2(P)>6$.

    \item If $d\ge4$ and $d$ is prime, then the $d$-pentasm is the unique minimiser of $f_k(P)$.
    \item If $d\ge4$ and $d$ is composite, then the minimiser of $f_k(P)$ is either a $d$-pentasm or a $(d-r-s)$-fold pyramid over $\Delta({r,s})$, where $d=rs$ is a nontrivial factorisation of $d$.
\end{enumerate}
\end{theorem}

The case $k=1$ of \ \cref{thm:2dplus1-bound} was  settled in \cite{PinUgoYosLBT}; we restate this result here as \cref{prop:edges}. As for which polytope is the actual minimiser in case (iii), we formulate a more precise conjecture at the  end of the paper. Generally speaking, the pentasm is the minimiser for low values of $k$.

\begin{figure}
    \centering
    \includegraphics[scale=.9]{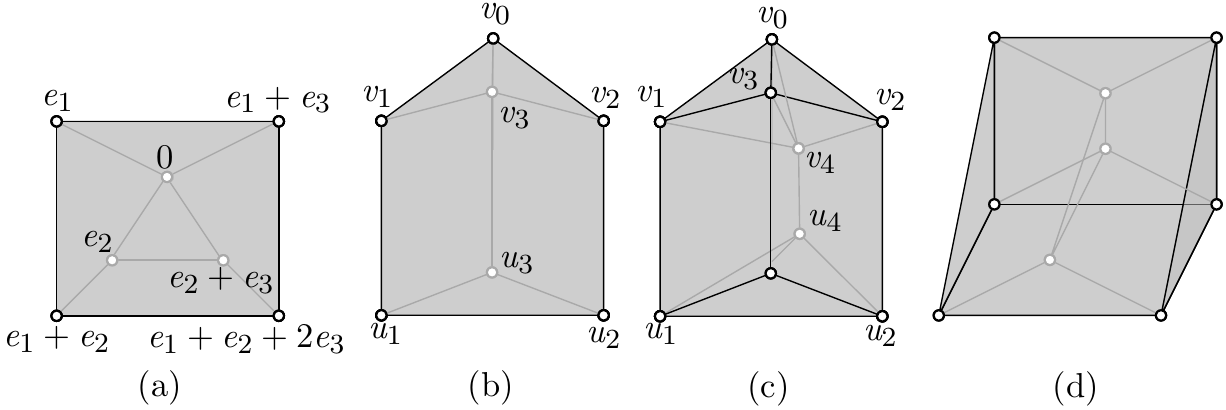}
    \caption{Schlegel diagrams of extremal polytopes with $2d+1$ vertices.  (a) $\Sigma(3)$. (b) 3-pentasm. (c) 4-pentasm. (d) $\Delta({2,2})$.}
    \label{fig:2dplus1}
\end{figure}

Let us now give a brief outline of the paper. In Sec. 2 we present some more background material: the concept of truncation, some essential examples of polytopes, previous lower bound bound theorems, and the fundamental technique of Xue \cite{Xue21}. The case $k=d-2$ of our main result needs to be established first, and so is the focus of Sec. 3. It requires the use of Gale diagrams, and leads to a result of independent interest: the characterisation of $d$-polytopes with $d+3$ vertices and at most $\h(d^2+5d-2)$ edges, i.e. two edges more than the known lower bound. This gives us the machinery we need to establish our main result in Sec. 4. The proof of the main result divides naturally into two cases: first for polytopes that have
a nonpyramidal, nonsimple vertex, and then for polytopes in which every vertex is either pyramidal or simple.

\section{Preliminaries}

Unless otherwise stated the notation and terminology follows  \cite{Zie95}. This section gathers a number of relatively recent results that we will need to prove our main theorem.

\subsection{Truncation of polytopes}\label{subsection:trunc}

Recall that a vertex in a $d$-polytope $P$ is  simple if and only if it is contained in exactly $d$ facets. A nonsimple vertex in $P$ may be simple in a proper face of  $P$; we often need to make this distinction. Indeed, every vertex is simple in every 2-face which contains it. For an edge $e=xy$ of any graph,  one says that the vertices $x$ and $y$ are \textit{adjacent} or \textit{neighbours}.

A fundamental tool for the construction of new polytopes is the truncation of a face \cite[p.~76]{Bro83}. An extension of this idea is the truncation of a set of vertices which do not necessarily form a face. This idea is implicit in \cite[Sec. 10.4]{Gru03}. Let $H$ be a hyperplane intersecting the interior of $P$ and containing no vertex of $P$, and let $H^+$ and $H^-$ be  the two closed half-spaces bounded by $H$. Set $P':=H^+\cap P$. Denoting by $X$ the set of vertices of $P$ lying in $H^-$,  the polytope $P'$ is said to be obtained by {\it truncating} the set $X$ by $H$. One often says that $P$ has been {\it sliced} or {\it cut} at $X$. We do not assume that $X$ forms the vertex set of any face of $P$. 

The next result establishes the simplicity of the ``new facet" of a truncated polytope, under quite general hypotheses.

\begin{proposition}\label{prop:trunc-neighbours-simple} Let $P$ be a $d$-polytope with vertex set $V$ and let $X\subset V$ be a collection of vertices for which the convex hulls of $X$ and $V\setminus X$ are disjoint. Suppose that $H$ is a hyperplane strictly separating $X$ and $V\setminus X$, and  that $H^+$ and $H^-$ are the two closed half-spaces bounded by $H$.
 Further suppose that for each edge $vw$ of $P$  with $v\in H^+$ and $w\in H^-$, at least one of the vertices $v,w$ is simple in $P$.

Denote by $P':=H^+\cap P$ the polytope obtained by truncating $X$ by $H$. Then every vertex in the facet $H\cap P$ is simple in $P'$, and thus the facet $H\cap P$ is a simple $(d-1)$-polytope.\end{proposition}

\begin{proof}
Any vertex $u$ in the facet $H\cap P$ is the intersection of $H$ and an edge $e$ of $P$ with one endpoint in $H^+$ and the other in $H^-$.  Since one of these endpoints is simple, it is contained in only $d$ facets of $P$, so the edge $e$ itself is contained in exactly $d-1$ facets of $P$.   Consequently, the facets of $P'$ containing the vertex $u$  are precisely the facets of $P$ containing the edge $e$ plus the facet $H\cap P$. Thus the vertex $u$ in $P'$ is contained in exactly $d$ facets of $P'$, and is simple in $P'$.
\end{proof}

With the notation of \cref{prop:trunc-neighbours-simple}, if $X=\{v\}$ then the facet $P\cap H$ of $P'$ is called the \textit{vertex figure} $P/v$ of $P$ at $v$. Then there is a bijection between the $k$-faces of $P$ that contains $v$, and the $(k-1)$-faces of $P/v$, \cite[Theorem 16]{McMShe71} or \cite[Prop.~2.4]{Zie95}. Some routine corollaries of  \cref{prop:trunc-neighbours-simple}  read as follows.

\begin{lemma}\label{lem:examples-all-neighbours-simple} Let $P$ be a $d$-polytope.
\begin{enumerate}
    \item Let $v$ be a simple vertex in $P$, and $P'$ the polytope obtained by truncating $v$. Then $f_0(P')=f_0(P)+d-1$ and $f_k(P')=f_k(P)+\binom{d}{k+1}$ for $k\ge1$.
    \item  If some 2-face of a simple 4-polytope $P$ has $m+1$ or more vertices, then there is another simple 4-polytope $P'$ with \[f_0(P')=f_0(P)+m+2\;\text{and}\; f_3(P')=f_3(P)+1.\]
\end{enumerate}
\end{lemma}

\begin{proof}
The proof of (i) is a  well-known direct consequence of \cref{prop:trunc-neighbours-simple}.

(ii) Let $X:=\left\{v_1,\ldots,v_m\right\}$ be a set of $m$ vertices in this 2-face  with $v_i$ adjacent to $v_{i+1}$ for each meaningful $i$. The  convex hull of $X$ is clearly disjoint from the convex hull  of the other vertices $Y$ of $P$, and so there is a hyperplane $H$ separating these two sets. We let $H^+$ be the closed halfspace determined by $H$ that contains $Y$ and let $P'$ be the polytope $P\cap H^+$.

Since $v_1$ and $v_m$ both have three neighbours in $Y$, while $v_2,\ldots,v_{m-1}$ each have two neighbours in $Y$, there are altogether $3\times 2+2(m-2)$ edges between $X$ and $Y$, whence there are $3\times 2+2(m-2)$ vertices in $P'$ that are not in $P$. In addition, the $m$ vertices of $X$  are not in $P'$, so $f_0(P')=f_0(P)+m+2$. Finally, we have that $f_3(P')=f_3(P)+1$.
\end{proof}

In particular, truncation of a simple polytope gives us another simple polytope with more vertices.

\subsection{Pentasms and polytopes with $d+2$ facets}\label{subsection:extremalexamples} The definition of truncation, together with \cref{prop:at-most-2d} and \cref{lem:examples-all-neighbours-simple}, easily yields the number of $k$-faces of a pentasm. For vertices we clearly have $f_0(Pm(d))=2d+1$. For higher dimensional faces

\begin{equation}
    \label{eq:pentasm}
    f_k(Pm(d))=
    \binom{d+1}{k+1}+\binom{d}{k+1}+\binom{d-1}{k},\qquad\text{if $k\ge 1$}.\\
\end{equation}

In one concrete realisation, the $d$-pentasm is the convex hull of the vectors ${0}$, ${e_i}$ for $i\in[1,d]$, and ${e_1}+{e_2}+{e_i}$ for  $i\in[1,d]$. The $d$-pentasm has $d+3$ simple vertices and $d-2$ vertices with degree $d+1$. The $d$-pentasm can also be described as the Minkowski sum of a $d$-simplex and a line segment that is parallel to one 2-face of the simplex but not parallel to any edge.
We list the facets of a pentasm in  more detail; they were first described in  \cite[Sec.~4]{PinUgoYosLBT}.
\begin{remark}\label{rem:pentasm-facets} We can label the vertices of any pentasm as $u_1,\ldots,u_d$, $v_0,v_1,\ldots,v_d$, in such a way that the edges are $u_iv_i$ for $1\le i\le d$, $u_iu_j$ for $1\le i<j\le d$, and $v_iv_j$ for $0\le i<j\le d$ except when $(i,j)=(1,2)$ (\cref{fig:2dplus1}(b)-(c)).
The $d$-dimensional pentasm has  precisely $d+3$ facets:
\begin{enumerate}
\item $d-2$ pentasms of lower dimension (for each $i\in[3,d]$, the face generated by all vertices except $u_i$ and $v_i$ is a pentasm facet),

\item two simplicial prisms (one generated by all vertices except $u_1,v_1,v_0$  and the other generated by all vertices except $u_2,v_2,v_0$),

\item and three simplices (one generated by all $u_i$, another generated by all $v_i$ except $v_1$, and the third generated by all $v_i$ except $v_2$.
\end{enumerate}
\end{remark}

In general, $d$-pentasms have the minimum number of edges of $d$-polytopes with $2d+1$ vertices.

\begin{proposition}
\cite[Theorem 14]{PinUgoYosLBT}\label{prop:edges}  The $d$-polytopes with $2d + 1$ vertices and $d^{2} + d-1$ or fewer edges are as follows.
\begin{enumerate}
\item  For $d = 3$, there are exactly two polyhedra with 7 vertices and 11 edges; $Pm(3)$ and $\Sigma(3)$. None
have fewer edges.
\item  For $d = 4$, a sum of two triangles $\Delta(2,2)$ is the unique polytope with 18 edges, and  $Pm(4)$ is the
unique polytope with 19 edges. None have fewer edges.
\item  For $d\ge 5$,  $Pm(d)$ is the unique $d$-polytope with $d^{2} + d-1$  edges. None have fewer edges.
\end{enumerate}

\end{proposition}

McMullen \cite[Sec.~3]{McM71}  provided expressions for  the number of $k$-faces of a $d$-polytope with $d+2$ facets.

\begin{lemma}[{\cite[Sec.~3]{McM71}}]\label{lem:dplus2facets}
Let  $P$ be a $d$-dimensional polytope with $d+2$ facets, where $d\ge 2$. Then, there exist $r>0$, $s>0$ and $t\ge 0$ such that $d=r+s+t$, and $P$ is a $t$-fold pyramid over {$\Delta({r,s})$}. In particular, every vertex of $P$ is either pyramidal or simple. The number of  $k$-dimensional faces of $P$ is
\begin{equation}\label{eq:dplus2facets}
f_k(P)= \binom{r+s+t+2}{k+2} -\binom{s+t+1}{k+2}-\binom{r+t+1}{k+2} +\binom{t+1}{k+2}.
\end{equation}
\end{lemma}
	
Conversely, any $d$-polytope of the form just described has precisely $d+2$ facets.	
We list them now.

\begin{remark}\label{rem:dplus2facets-facets}
The $r+s+t+2$ facets of a $t$-fold pyramid over $\Delta(r,s)$ are as follows:
\begin{enumerate}
\item $r+1$ facets that are  $t$-fold pyramids over $\Delta(r-1,s)$,

\item $s+1$ facets that are $t$-fold pyramids over $\Delta(r,s-1)$,

\item and $t$ facets that are $(t-1)$-fold pyramids over $\Delta(r,s)$.
\end{enumerate}
\end{remark}

McMullen actually characterised the values of $f_0$ for which there is a $d$-polytope with $d+2$ facets and $f_0$ vertices; we  state a special case.

\begin{corollary}[{\cite[special case of Thm.~2]{McM71}}]\label{cor:dprime}
If $d$ is prime then there is no  $d$-polytope with $d+2$ facets and $2d+1$ vertices.
\end{corollary}

\subsection{Capped prisms}
\label{subsection:capped}
We define another family of polytopes related to the pentasms. For $3\le \ell\le d$, let $Q_{\ell}$ be a bipyramid over an $(\ell-1)$-simplex, and let $P_{\ell,d}$ be a $(d-\ell)$-fold pyramid over $Q_{\ell}$. Then $P_{\ell,d}$ is a $d$-polytope with $d+2$ vertices, only two of which are simple, say $v_{0}$ and $v_{1}$. Truncating one simple vertex from $P_{\ell,d}$, say $v_{1}$, gives us a {\it capped $d$-prism}, which we denote  by $CP(\ell,d)$. This definition is consistent with the one given in \cite[Remark 2.11]{PinUgoYosEXC}. All of the $d-2$ polytopes constructed in this manner have the same graph. In particular they have $2d+1$ vertices and $d^2+d$ edges, one more than the pentasm. Nevertheless their combinatorial type depends on the value of $\ell$;  $CP(\ell,d)$ has $d+\ell+1$ facets, so their $f$-vectors  are all distinct. (For $\ell=2$, the same construction leads to the $d$-pentasm, with $d^2+d-1$ edges.)

The aforementioned vertex $v_0$ remains a vertex of $CP(\ell,d)$. It is also the case that the convex hull of the other $2d$ vertices in $CP(\ell,d)$ is a simplicial $d$-prism, and $v_0$ is beyond one simplex facet of this prism, and not beyond any of the other facets; hence the name capped prism. A point is said to be {\it beyond} a facet \cite[Sec. 5.2]{Gru03} of a polytope if the supporting hyperplane for that facet strictly separates the point from the polytope.

\begin{lemma}\label{lem:capped}
If $3\le \ell\le d$, then, for each $1\le k\le d-1$, we have that \[f_k(CP(\ell,d))>f_k(Pm(d)).\]
\end{lemma}

\begin{proof} A $d$-pentasm is obtained by truncating a simplex vertex from the triplex $M(2,d-2)$ and a capped prism is obtained by truncating a simplex vertex from the aforementioned $P_{\ell,d}$.
Since $P_{\ell,d}$ has $d+2$ vertices, it has strictly more faces of all nonzero dimensions than $M(2,d-2)$, by \cref{prop:at-most-2d}. \cref{lem:examples-all-neighbours-simple}(i) informs us that this inequality persists after truncating a simple vertex,  thereby establishing this lemma.
\end{proof}

\subsection{The result of Xue \cite{Xue21}} Let $f_k(F,P)$ be the number of  $k$-faces in $P$ containing the face $F$. The result \cite[Prop.~3.1]{Xue21}, originally designed for $d$-polytopes with up to $2d$ vertices, is key to our work. Its proof in \cite{Xue21} actually establishes something slightly more general. Let $\dim P$ denote the dimension of a polytope $P$.

\begin{proposition}[{\cite[Prop.~3.1]{Xue21}}]\label{prop:number-faces}
Let $d\ge 2$ and let $P$ be a $d$-polytope. In addition, suppose that  $r\le d+1$ is given and  that $S:=(v_1,v_2,\ldots,v_r)$ is a sequence of distinct vertices in $P$. Then the following hold.
\begin{enumerate}
    \item  There is a sequence $F_1, F_2,\ldots, F_r$ of faces of $P$ such that each $F_i$ has dimension $d-i+1$ and contains $v_i$, but does not contain any $v_j$ with $j<i$.
    \item Suppose $\deg_{F_{i}} (v_i)\le 2(\dim F_i-1)$ for each $i\in [1,\min\{r,d-1\}]$. Then for each $k\ge2$, the number of $k$-faces of $P$ that contain at least one of the vertices in $S$ is bounded from below by \[\sum_{i=1}^{r}\phi_{k-1}(\deg_{F_i} (v_i),\dim  F_i-1).\]
\end{enumerate}
\end{proposition}
\begin{proof}[Proof sketch] The proof of (i) is as in the proof of \cite[Prop.~3.1]{Xue21}. The extension to the case $r=d+1$ is easy; take $F_{d+1}=\{v_{d+1}\}$. For the veracity of (ii) observe that the number $f_k(v_i,F_i)$ of $k$-faces in $F_i$ containing $v_i$ equates to the number of $(k-1)$-faces in the vertex figure $F_i/v_i$ of $F_i$ at  $v_i$. Now \[f_0(F_i/v_i)=\deg_{F_{i}} (v_i)\le 2(\dim F_i-1),\] and so the lower bound $\phi_{k-1}(\deg_{F_i} (v_i),\dim  F_i-1)$ of  \cref{prop:at-most-2d} applies to $f_{k-1}(F_i/v_i)$.
In case that $i=d$ or $d+1$, then $\dim  F_i\le1$, and for each $k\ge 2$ we have that \[f_{k}(v_i,F_i)=0=\phi_{k-1}(\deg_{F_i} (v_i),\dim  F_i-1).\]   Consequently, for each $k\ge 2$ we get that
\[\sum_{i=1}^{r} f_{k}(v_i,F_i)= \sum_{i=1}^{r} f_{k-1}(F_i/v_i)\ge \sum_{i=1}^{r}\phi_{k-1}(\deg_{F_i} (v_i),\dim  F_i-1).\]
The sketch of the proof is now complete.
\end{proof}

Several instances of \cref{prop:number-faces} are of particular interest here, and so we gather them in the next corollary.

\begin{corollary}\label{cor:number-faces}
Let $P$ be a $d$-polytope with $2d+1$ vertices,  let $r\le d$,  and let $S$ be a sequence of $r$ vertices in $P$. Then  the number of $k$-faces of $P$ that contain at least one of the vertices in $S$ is at least
\begin{enumerate}
\item $\displaystyle\phi_{k-1}(\deg_{P}(v),d-1)+\sum_{i=2}^{r} \binom{d-i+1}{k},\; \text{for any chosen vertex $v\in S$}.$

\item  This expression is bounded from below by
$$
\begin{cases}
\displaystyle\binom{d}{k}+\binom{d-1}{k}-\binom{d-2}{k}+\sum_{i=2}^{r} \binom{d-i+1}{k}, & \text{if $v$ is nonsimple,}\\
\displaystyle\sum_{i=1}^{r} \binom{d-i+1}{k}, & \text{in any case.}
\end{cases}
$$\end{enumerate}

\end{corollary}

\section{Polytopes with few edges}

The case $k=d-2$ of our main result, i.e. the case of ridges, was announced without proof in \cite[Prop. 30]{PinUgoYosLBT}. It is more natural to prove first the dual result, which characterises certain polytopes with a close to minimal number of edges. The combinatorial structure of all $d$-polytopes with $d+2$ vertices  was characterised in \cite[Chap. 10]{Gru03}. Such polytopes have at least $\phi(d+2,d)={\binom{d}{2}}-2$ edges, and the only one with exactly $\phi(d+2,d)$ edges is the triplex $M(2,d-2)$.

We will establish a similar result for $d$-polytopes with $d+3$ vertices. We know that the only one with exactly $\phi(d+3,d)$ edges is the triplex $M(3,d-3)$.  Furthermore, the only ones with exactly $\phi(d+3,d)+1$ edges are the $(d-2)-$fold pyramid over the pentagon, and the $(d-3)-$fold pyramid over the tetragonal antiwedge; this is easily deduced from  Gale diagram techniques like those about to be presented  (see also \cite[Theorem 3.3]{PinUgoYosCSimp}). (The \textit{tetragonal antiwedge} is the unique nonpyramidal 3-polytope with six vertices and six faces; four of the faces are triangles and two are quadrangles.) We shall completely characterise all such polytopes with $\phi(d+3,d)+2$ edges, at least in terms of their Gale diagrams.

The structure of polytopes with $\phi(d+3,d)+3$ or more edges is of course more complicated; we refer to  \cite{Fus06} for more details. Complete catalogues of $d$-polytopes with $d+3$ vertices have been constructed only for $d\le6$; see \cite{FukMiyMor13}.

A vertex $v$ is \textit{pyramidal} in a polytope $P$ if $P$ is a pyramid with apex $v$; otherwise the vertex is \textit{nonpyramidal}. Alternatively, a vertex $v$ is pyramidal in $P$ if $P$ is the convex hull of $v$ and a $(d-1)$-polytope that does not contain $v$, the \textit {base} of the pyramid.

We assume familiarity with Gale diagrams as in \cite[Sec. 6.3]{Gru03},\cite{McMShe71} or \cite[Sec. 6.4]{Zie95}. For a $d$-polytope with $d+s$ vertices, a standard Gale transform is a mapping from the vertex set into $\mathbb{R}^{s-1}$, which sends non-pyramidal vertices to the unit sphere, while pyramidal vertices are transformed to the origin. Transforms of different vertices may be located at the same point. The Gale diagram is this collection of transformed points. When $s\le d$, the Gale transform has lower dimension than the original polytope, and so is easier to visualise. This is particularly so when $s=2$ or 3. Most importantly, the Gale diagram contains complete information about the face lattice of the original polytope. This compact representation makes the Gale diagram a powerful tool for studying the combinatorial types of polytopes with few vertices. For convenience, we will continue to denote vertices of a polytope by lower case letters, but we will now denote their transformed points in the Gale diagram by the corresponding upper case letters.

Let us explain briefly how to extract information about the face lattice from Gale diagrams. We start with the fundamental property that any open hemisphere must contain the transforms of at least two vertices. A {\it coface} of a polytope is a collection of points in its Gale diagram whose convex hull contains the origin in its relative interior. It transpires that a collection of vertices in a polytope is the vertex set of a face if, and only if, the transformed points in the Gale diagram are the complement of a coface. This makes it relatively easy to read off  the vertex-facet incidence relations from the Gale diagram; one just has to identify the minimal cofaces. For example, if two vertices $u$ and $v$ are such that their transforms $U,V$ are diametrically opposite, then there is a facet which contains all vertices except $u$ and $v$. On the other hand, if a collection of three or more vertices have their Gale transforms at the same point on the sphere, then they must all be adjacent to one another, since, for every pair of them,  the convex hull of the  transforms of the remaining vertices will be the same.

We are only interested here in polytopes with $d+3$ vertices, so the Gale diagram is two-dimensional, and the transforms of non-pyramidal points are contained in the circle $S^1$.  This makes it easy to read off the missing edges from the Gale diagram; by \textit{missing edge} we mean a pair of nonadjacent vertices.  Let us introduce some temporary vocabulary to streamline our arguments in this section. Any two non-diametral points will divide the circle into two arcs, one of which contains a semicircle, and one of which does not; the latter will be called the {\it short arc} between the two points. Say that two points on a Gale diagram are \textit{contiguous} if no other transformed points lie on the short arc between them. This includes the case of two points at the same location, but not three. Diametrically opposite points are not  contiguous; both semicircles between them must contain other points. The point  diametrically opposite a transformed vertex $V$ will be denoted $-V$; it need not be the transform of any vertex. The next lemma gathers a number of results that routinely follow from the discussion in \cite[Sec. 6.3]{Gru03}.

\begin{lemma}\label{lem:galefacts}
Let $t,u,v\ldots$ be  vertices of  a  $d$-polytope $P$ with $d+3$ vertices, and let  $T,U,V\ldots$ be their  Gale transforms.
\begin{enumerate}
\item if  $V$ is contiguous with both $U$ and $W$, then $u$ and $w$ must be adjacent in $P$.
\item if $v$ and $w$ are not adjacent, then $V$ and $W$ must be contiguous. If $u$ is another vertex with $U$ contiguous to $V$ (on the other side  from $W$), then the Gale transforms of all other vertices must be contained in a closed semicircle from $W$ to $-W$.
\item if $v$ is a simple vertex, then there are precisely two vertices $u,w$ whose transforms $U,W$ are contiguous to $V$, and the transforms of all other vertices are (on the short arc) between $-U$ and $-W$.
\item  if $t,u,v,w$ are four vertices of $P$ whose transforms $T,U,V,W$ are ordered cyclically on the circle, so that cyclically consecutive pairs are contiguous, and $u,v$ is a missing edge of $P$, then the arc from $W$ to $T$ (in the given orientation) is the short one.
\end{enumerate}
\end{lemma}

We now need the concept of  {\it free join} of two polytopes. Recall that two affine spaces are {\it skew} if they do not intersect, and no line from one space is parallel to a line from the other.
Let  $P_{1}\subset \R^{d_{1}+d_{2}+1}$ be a $d_{1}$-polytope and $P_{2}\subset\R^{d_{1}+d_{2}+1}$ a $d_{2}$-polytope such that their affine hulls are skew; then the {\it free join} of the  polytopes $P_{1}$ and $P_{2}$  is the $(d_{1}+d_{2}+1)$-polytope $P:=\conv (P_{1}\cup P_{2})$. The $k$-faces of the free join are as one would expect; see \cite[Exercise 4.8.1]{Gru03} or \cite[Exercise 9.9]{Zie95}.

\begin{lemma}
\label{lem:join-faces} Let $P_{1}$ and $P_{2}$ be two polytopes, and let  $P$ be the free join of $P_{1}$ and $P_{2}$. The $k$-faces of $P$ are precisely the free joins of a face $F_{1}$ of $P_{1}$ and a face $F_{2}$ of $P_{2}$ such that $\dim F_{1}+\dim F_{2}+1=k$.
\end{lemma}

We are now able to give the promised characterisation of $d$-polytopes with $d+3$ vertices and only two more edges than the corresponding triplex.

\begin{proposition}
\label{prop:d+3}
Let $P$ be  a nonpyramidal $d$-polytope with $d+3$ vertices and exactly $\phi(d+3,d)+2$ edges. Then its standard contracted Gale diagram is one of the six indicated in \cref{fig:dplus3}. The affix $d-2$ indicates that $d-2$ vertices of $P$ have their transforms located at the same point on the circle.
\begin{enumerate}
   \item The polytope represented by \cref{fig:dplus3}(i) exists  only in dimension $d=3$; it has 7 facets, and is the dual of $\Sigma(3)$;

\item The polytope in \cref{fig:dplus3}(ii) exists in any dimension $d\ge3$; it has $2d+1$ facets, and is the dual of a pentasm;

\item The polytope represented by \cref{fig:dplus3}(iii) exists in any dimension $d\ge4$; it has $2d$ facets;

    \item The polytope represented by \cref{fig:dplus3}(iv) exists in any dimension $d\ge4$; it has $2d-1$ facets;

    \item  The polytope represented by \cref{fig:dplus3}(v) exists  only in dimension $d=4$; it has $8$ facets, and corresponds to the second diagram in \cite[Fig. 6.3.4]{Gru03};

  \item  The polytope represented by \cref{fig:dplus3}(vi) exists only in dimension $d=5$; it has $8$ facets, and is the free join of two quadrilaterals.
\end{enumerate}

\end{proposition}

\begin{figure}
    \centering
    \includegraphics[scale=0.8]{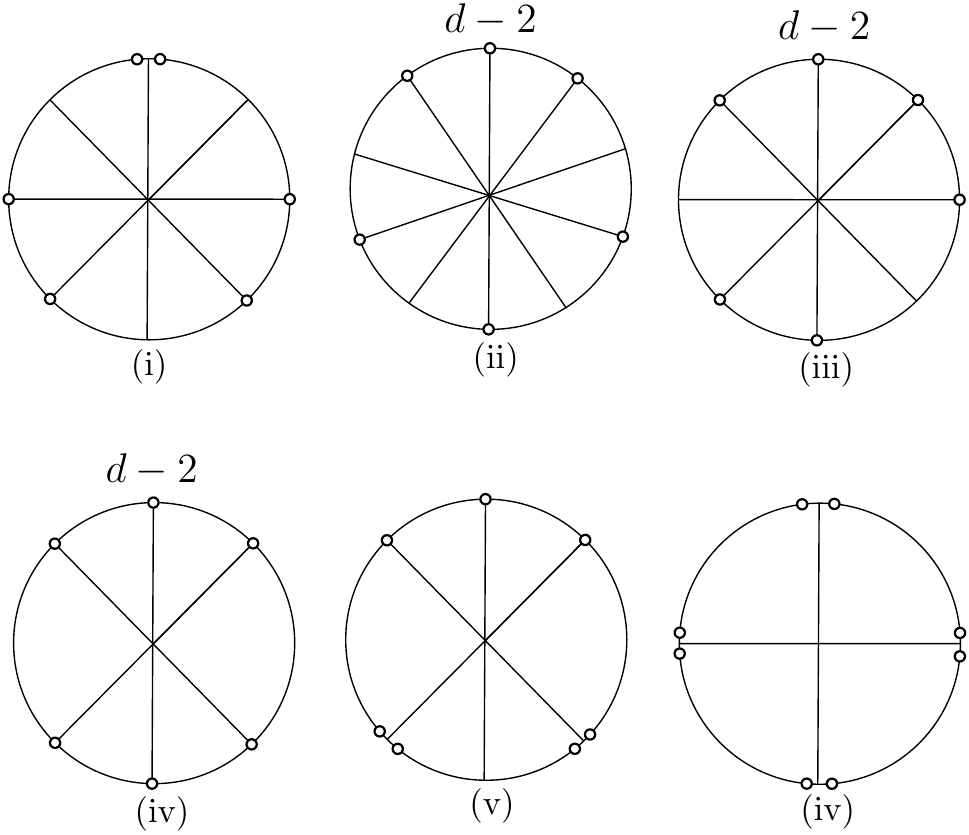}
    \caption{Gale diagrams of nonpyramidal $d$-polytopes with $d+3$ vertices and $\phi(d+3,d)+2$ edges.}
    \label{fig:dplus3}
\end{figure}

\begin{proof}
The graph of $P$ has exactly four missing edges. We consider its complement graph, which is a graph containing just four edges (and a number of isolated vertices). It  follows from \cref{lem:galefacts}(i)-(ii) that these cannot form  a cycle.

First, suppose that the complement of the graph contains a path of length three, but not four. Label the vertices in this path, in order, as $t, u, v,w$. Since $t,u$ and $v,w$ are missing edges, all points in the Gale diagram must be between $-U$ and $-V$. Since $u,v$ is a missing edge, $T$ and $W$ must lie in a closed semicircle not containing $U$ or $V$. The existence of an additional two vertices vertices constituting a fourth missing  edge implies that $T,U,V,W$ all lie in a closed semicircle, which ensures that $T$ and $W$ are diametrically opposite. It follows that there are only six vertices altogether, so $d=3$, and $P$ is in fact the dual of $\Sigma(3)$. This is case (i). We remark that the two other transformed points must be strictly between $-U$ and $-V$, otherwise either $t$ or $w$ would belong to an additional missing edge.

Next, suppose they form a path of length four. We label the five vertices as $t,u,v,w,x$ so that each successive pair is a missing edge of $P$.  Applying \cref{lem:galefacts}(ii) to both the missing edges $t,u$ and $w,x$ ensures that the Gale transforms of all other vertices are located at $-V$. The next question is: how many diametrically opposite pairs there are, amongst $T,U,V,W,X$? If there are none, we are in case (ii), which is easily checked to be the dual of a pentasm: the $d-2$ co-located points correspond to the pentasm facets, the two upper points in the diagram to the simplicial prism facets, and the lower three points to the simplex facets.

If there is one diametrically opposite pair, we are in case (iii). If there are two, we are in case (iv). Note that we must assume $d\ge4$ in these two cases. (If $d=3$, these are Gale diagrams of polytopes, but (ii) would represent a tetragonal antiwedge, which has five missing edges, while (iii) would represent a simplicial prism, which has six missing edges.)

Every facet of $P$ has either $d$ or $d+1$ vertices, so every minimal proper coface contains either two or three transformed vertices. This makes it easy to read off the  minimal cofaces from the  Gale diagram. In case (iii),  there are $2d$ of them, $d-1$ diameters and $d+1$ which contain three points. In case (iv),  there are $2d-1$ minimal cofaces, $d$ diameters and $d-1$ which contain three elements.

Now suppose that the complement of the graph contains a path of length two, but not three. Denote by $u,v$ and $v,w$ the corresponding missing edges. Then the transforms of every other vertex must lie on the arc from $-U$ to $-W$. Let $a,b$ be another missing edge and suppose that one (or both) of $A,B$ is not located at $-U$. Then all transformed points other than $A,B$ must lie in the closed semicircle from $U$ to $-U$; in particular all transformed points other than $A,B,U,V,W$ must be located at $-U$. Repeating this argument for the other missing edge, we see that there can only be seven points altogether, with two located at $-U$ and two located at $-W$. The only possibility is (v), so $d=4$. For the record, its facets are four quadrilateral pyramids and four simplices.

Finally suppose that the complement of the graph contains four disjoint edges. For each such pair of vertices, the transforms of all other vertices must lie in a closed semicircle. Similar reasoning shows  there can only be these eight vertices, so $d=5$ and the Gale diagram (vi) is the only possibility.
\end{proof}

We next establish the main result for the special case $k=d-2$: a $d$-polytope with $2d+1$ vertices and no more ridges than the $d$-pentasm must either be a pentasm, or have $d+2$ facets, or be 3-dimensional.

\begin{proposition}\label{prop:ridges}
Let $P$ be any $d$-polytope.
\begin{enumerate}
    \item If $P$ has at least $d+3$ vertices, at most $\phi(d+3,d)+2$ edges and exactly $2d+1$ facets, then $P$ must be the dual of a pentasm or the dual of $\Sigma(3)$.
    \item A $d$-polytope with  $2d+1$ vertices, at least $d+3$ facets, and no more $(d-2)$-faces than the pentasm must be a pentasm or  $\Sigma(3)$.
\end{enumerate}
\end{proposition}

\begin{proof}
(i) If $d=3$, the hypotheses imply that $f(P)=(6,11,7)$, and the conclusion is easy to check, directly or from catalogues e.g. \cite{BriDun73}. Henceforth we assume that $d\ge4$.

A $d$-polytope with $2d+1$ or more vertices has at least $\h d(2d+1)>\phi_1(d+3,d)+2$ edges. And a $d$-polytope with $d+s$ vertices, where $d\ge s\ge4$, has at least $\phi_1(d+s,d)$ edges. The definition \eqref{eq:at-most-2d} yields that  $\phi_1(d+s,d)$  is an increasing function of $s$, and so
$$\phi_1(d+s,d)\ge \phi_1(d+4,d)>\phi_1(d+3,d)+2$$
for $s\ge4$. Thus a $d$-polytope with $d+s$ vertices, where $d\ge s\ge4$, has more than  $\phi_1(d+3,d)+2$ edges. This leaves only the case when $P$ has exactly $d+3$ vertices.

If $P$ has $\phi(d+3,d)$ edges or $\phi(d+3,d)+1$ edges, then it is a multifold pyramid over a simplicial 3-prism,  a pentagon, or a tetragonal antiwedge, and these respectively have $d+2$, $d+3$, or $d+3$ facets, less than $2d+1$.

\cref{prop:d+3} informs us that, apart from pentasms and $\Sigma(3)$, every nonpyramidal $d$-polytope with $d+3$ vertices and $\phi_1(d+3,d)+2$ edges has strictly less than $2d+1$ facets. Forming a $k$-fold pyramid increases both the dimension and the number of vertices by $k$. The only possibility to have exactly $2d+1$ facets is then to be nonpyramidal, and to be in case (i) or (ii) of the preceding lemma.

(ii) Since $f_{d-2}(Pm(d))=\h(d^2+5d-2)=\phi_1(d+3,d)+2$, the conclusion follows by duality from (i).
\end{proof}

\section{Main result}

Recall our aim of showing that for $d\ge4$, $d$-pentasms and $d$-polytopes with $d+2$ facets are the extremal $d$-polytopes with $2d+1$ vertices.

 We make repeated use of the following three elementary identities, which can be found in numerous sources, e.g. \cite[p.~174]{GKP94}.
\begin{align}
    (a-b)\binom{a}{b}&=a\binom{a-1}{b}\label{eq:binomial1}\\
    \binom{a}{b}&=\binom{a-1}{b-1}+\binom{a-1}{b}\label{eq:binomial2}\\
    \sum_{j=0}^{d}\binom{j}{k}&=\binom{d+1}{k+1}\label{eq:binomial3}
\end{align}

\begin{lemma}
\label{lem:pyramids-dplus3-facets} Fix $t\ge 0$ and $d\ge 3$. Let $P$ be a $d$-polytope with at least $2d+1$ vertices and at least $d+3$ facets. If $P$ is a $t$-fold pyramid over a simple polytope, then
\begin{enumerate}
\item    $ f_k(P)\ge f_k(Pm(d))$, if $k=0$ or $k=d-1$;
\item     $f_k(P)> f_k(Pm(d))$, if $k\in [1,d-2]$.
\end{enumerate}
\end{lemma}
\begin{proof} The statement (i) is obvious from the hypotheses. Now fix $k\in[1,d-2]$.

We first establish part (ii) when $t=0$. In this case, $P$ is a simple $d$-polytope with at least $d+3$ facets. By the lower bound theorem for simple polytopes (\cref{prop:simple}), we obtain that
\begin{align*}
f_{k}(P)&\ge \binom{d}{k+1}f_{d-1} -\binom{d+1}{k+1}(d-1-k).
\end{align*}

Since $f_{d-1}(P)\ge d+3$ we also have that
\begin{align*}
    f_{k}(P) &\ge \binom{d}{k+1}(d+3) -\binom{d+1}{k+1}(d-1-k)\\
 &=\binom{d}{k+1}(d+3) -\binom{d+1}{k+1}(d-k)+\binom{d+1}{k+1} \\
  &=\binom{d}{k+1}(d+3) -\binom{d}{k+1}(d+1)+\binom{d+1}{k+1}\; \text{by \eqref{eq:binomial1}}\\
   &=\binom{d-1}{k}+\binom{d-1}{k+1}+\binom{d}{k+1}+\binom{d+1}{k+1}\; \text{by \eqref{eq:binomial2}}\\
   &=\binom{d-1}{k+1}+f_{k}(Pm(d))\; \text{by \eqref{eq:pentasm}}\\
   &>f_{k}(Pm(d)).
\end{align*}

Now consider the case $d=3$. We have established this for $t=0$, and the case $t=2$ does not arise, because a tetrahedron has $4<2d+1$ vertices.
For $t=1$, $P$ is a pyramid over an $n$-gon, with $n\ge6$, in which case, we have that $f_1(P)=2n>11=f_1(Pm(3))$. Hence we have the basis case for an inductive argument on $d$. Assume $d\ge 4$.

We may suppose that $t\ge 1$. Then $P$ is a  pyramid over a $(d-1)$-polytope $F$, itself a $(t-1)$-fold pyramid over a simple polytope. Moreover, $F$ has at least $2(d-1)+2$ vertices and at least $(d-1)+3$ facets, and our induction hypothesis applies to $F$. From $P$ being a pyramid over $F$ it follows that
\begin{equation}\label{eq:pyramid-formula}
    f_k(P)=f_k(F)+f_{k-1}(F).
\end{equation}
In  case  $k=1$, $f_1(F)>f_1(Pm(d-1))=(d-1)^2+(d-1)-1$ by the induction hypothesis, and $f_{0}(F)=f_{0}(P)-1$. Thus, from   \eqref{eq:pyramid-formula} we obtain that
\[f_{1}(P)>(d-1)^2+(d-1)-1+2d=d^{2}+d-1=f_{1}(Pm(d)).\]
In  case $k\in[2,d-2]$, the induction hypothesis and \eqref{eq:pyramid-formula}  ensure that
\begin{align*}
f_{k}(P)&=f_{k}(F)+f_{k-1}(F)>f_{k}(Pm(d-1))+f_{k-1}(Pm(d-1))=f_{k}(Pm(d)).
\end{align*}
The last equation is a straightforward application of \eqref{eq:binomial2} to \eqref{eq:pentasm}. We thus have that $f_{k}(P)>f_{k}(Pm(d))$ for $t\ge 1$ and each $k\in[1,d-2]$.
The lemma is now proved.
\end{proof}

A significant corollary of \cref{lem:pyramids-dplus3-facets} is that if a $d$-polytope with $2d+1$ vertices has no more $k$-faces than the pentasm for one value of $k$, and every vertex is either pyramidal or simple, then the polytope must have $d+2$ facets.

\begin{corollary}\label{cor:pyramids-dplus2-facets}
Let $d\ge 3$, and let $P$ be a $d$-polytope with at least $2d+1$ vertices. If $P$ is a multifold pyramid over a simple polytope such that $f_{k}(P)\le f_{k}(Pm(d))$  for some $k\in[1,d-2]$, then $P$ has $d+2$  facets, in which case $P$ is a $t$-fold pyramid over $\Delta(r,s)$ for suitable $r,s,t$ (in particular, $d=r+s+t$).
\end{corollary}

By virtue of \cref{cor:pyramids-dplus2-facets}, an extremal $d$-polytope $P$ with $2d+1$ vertices other than a multifold pyramid over $\Delta(r,s)$  must have a nonpyramidal, nonsimple vertex in $P$. Our strategy is to divide the problem of finding the $d$-polytopes with $2d+1$ vertices and minimum number of faces into two parts: first for polytopes that have at least one nonpyramidal, nonsimple vertex, and then for polytopes in which every vertex is either pyramidal or simple. In the first case, the only minimisers turn out to be pentasms; in the second case, polytopes with $d+2$ facets also come into play.

Our main result has an inductive step, which runs more smoothly if we present the low-dimensional cases first. Most of the next lemma is routine, but we prove everything for the sake of clarity.

\begin{lemma}\label{lem:five-dim}
Consider a $d$-polytope $P$ with $2d+1$ vertices, where $d\le5$.
\begin{enumerate}
    \item Let $d=3$. If $P$ is $\Sigma(3)$ or a pentasm, then $f(P)=(7,11,6)$. Otherwise, $f_1(P)>11$ and $f_2(P)>6$.
    \item Let $d=4$. If $P$ is $\Delta(2,2)$, then $f(P)=(9,18,15,6)$. If $P$ is  a pentasm, then $f(P)=(9,19,17,7)$. Otherwise, $f_1(P)\ge 20$, $f_2(P)\ge 18$, and $f_3(P)\ge7$.
    \item Let $d=5$. If $P$ is  a pentasm, then $f(P)=(11,29,36,24,8)$. Otherwise, $f_1(P)\ge 30$, $f_2(P)\ge 38$, $f_3(P)\ge 25$, and $f_4(P)\ge8$.
\end{enumerate}
\end{lemma}

\begin{proof}
(i) This is an easy exercise from Steinitz's theorem. One may also consult catalogues such as \cite[Fig.~4]{BriDun73}.

(ii) The $f$-vectors of $\Delta(2,2)$ and the 4-pentasm are easy to verify. Any 4-polytope with $f_0=9$  and $f_3=6$  must have $f_1\ge18$ and $f_2\le15$, and by Euler's relation we know that $f_1=f_2+3$. The only possibility is the simple polytope $\Delta(2,2)$. If $P$ is any 4-polytope with 9 vertices other than these two, then clearly $f_3\ge7$, and \cref{prop:edges} implies that $f_1\ge20$. Thus $f_2=f_1+f_3-9\ge18$.

These bounds are tight: a pyramid over a cube has $f$-vector $(9,20,18,7)$.

(iii) Again, the $f$-vector of the pentasm is clear. Suppose $P$ has 11 vertices but is not a pentasm.  \cref{prop:edges} then informs us that $f_1\ge30$, and $d$ being prime means $f_4\ge8$ by \cref{cor:dprime}.

Next we show that $f_{3}\ge 25$; the proof depends on the value of $f_{4}$. If $f_4=8$, \cref{prop:ridges}(ii) ensures that $f_3\ge25$. We claim that if $f_4=9$, then in fact $f_3\ge26$. According to \cite[Theorem 19]{PinUgoYosLBT}, the only 5-polytope with 9 vertices and 25 or fewer edges is the triplex $M(4,1)$, which has $7\ne11$ facets. The claim follows by duality. Next we claim that if $f_4=10$, then $f_3\ge27$. Again, by \cite[Theorem 19]{PinUgoYosLBT}, the only 5-polytope with 10 vertices and 26 or fewer edges is the simplicial prism $M(5,0)$, which has $7\ne11$ facets. The claim follows by duality. In the case that $f_4\ge11$, we have $2f_3\ge5f_4\ge55$, so $f_3\ge28$. So we have that $f_3\ge25$ in all cases.

We next claim that $f_3-f_4\ge17$. This is clear if $f_4=8,9,10 $ or 11. If $f_4\ge12$, we have $f_3-f_4\ge{\frac52}f_4-f_4\ge18$ as required. Finally \[f_2=f_1+f_3-f_4-f_0+2\ge30+17-11+2=38.\]
This completes the proof of the lemma.
\end{proof}

We recall that there are four polytopes with $f$-vector $(10,21,18,7)$ \cite[Sec. 6]{PinUgoYosEXC}; a pyramid over any of them will have $f$-vector $(11,31,39,25,8)$. Moreover a 5-dimensional capped prism has $f_0=11$ and $f_1=30$. Thus, in \cref{lem:five-dim}(iii), the bounds for $f_1,f_3$ and $f_4$ are tight. However the bound for $f_2$ is not. It can also be proved that there is no 5-polytope with $f_0=11$ and $f_2=38$, but the proof of this is long, and its inclusion would be an unnecessary digression from the aim of this paper.

Our next step is to show that our minimising polytopes have very restricted facial structure.

\begin{lemma}\label{lem:nonpyramidal} Fix $d\ge 6$, and let $P$ be a $d$-polytope with $2d+1$ vertices and with a nonpyramidal, nonsimple vertex. Suppose $f_k(P)\le f_k(Pm(d))$ for some $k\in [2,d-2]$. If $F$ is a facet avoiding at least one nonpyramidal nonsimple vertex, then either
\begin{enumerate}
\item   $F$   is a triplex $M(2,d-3)$ and every vertex outside $F$ has degree at most $d+1$; or
\item   $F$  is a simplicial $d$-prism $M(d-1,0)$ and every vertex outside $F$ has degree at most $d+1$; or
\item   $F$  contains exactly $2d-1$ vertices, one (or more) of which is nonpyramidal and nonsimple in $F$; consequently F has at least $d -1+3$ facets (i.e. $F$ contains $d+2$ ridges of $P$); or
\item $F$ is a simplex, but there is another facet $F'$ that is not a simplex and also avoids at least one nonpyramidal nonsimple vertex; hence $F'$ falls in one of the previous cases (i)--(iii).
\end{enumerate}
\end{lemma}

\begin{proof} Choose a nonpyramidal nonsimple vertex, and let $F$ be a facet of $P$ not containing it. The facet $F$ is not the base of a pyramid, so we may suppose that $F$ has exactly $d-1+d+2-r$ vertices with $r\in[2,d+1]$. There are $r$ vertices outside $F$, say $S:=\left\{v_{1},v_{2},\ldots,v_{r}\right\}$, none of them pyramidal. Without loss of generality, suppose that $v_{1}$ has the largest possible degree among the vertices in $S$; clearly $v_1$ is nonsimple.

(i)-(ii) Consider the case $r\in [3,d]$. Since $d+2-r\le d-1$, it follows from \cref{prop:at-most-2d} that  $f_{k}(F)\ge \phi_{k}(d-1+d+2-r,d-1)$ for $k\in[1,d-2]$. The $k$-faces in $P$ consist of the $k$-faces of $F$ plus the $k$-faces containing at least one of the vertices in $S$. \cref{cor:number-faces}(ii) then informs us that
\begin{align}
f_{k}(P)&\ge f_{k}(F)+ \phi_{k-1}(\deg_{P}(v_{1}),d-1)+\sum_{i=2}^{r} \binom{d-i+1}{k}\nonumber\\
&\ge \phi_{k}(d-1+d+2-r,d-1)+ \phi_{k-1}(d+1,d-1)+\sum_{i=2}^{r} \binom{d-i+1}{k}\label{eq:nonpyramidal1}\\
&= \binom{d}{k+1}+\binom{d-1}{k+1}-\binom{r-2}{k+1}+ \binom{d}{k}+\binom{d-1}{k}-\binom{d-2}{k}+\sum_{i=2}^{r} \binom{d-i+1}{k}\nonumber\\
&= \binom{d+1}{k+1}+\binom{d}{k+1}-\binom{r-2}{k+1}-\binom{d-2}{k}+\sum_{i=2}^{r} \binom{d-i+1}{k}\nonumber\\
&= f_{k}(Pm(d))-\binom{r-2}{k+1}+\sum_{i=4}^{r} \binom{d-i+1}{k}\nonumber\\
&= f_{k}(Pm(d))-\binom{r-2}{k+1}+ \sum_{j=0}^{r-4}\binom{d-3-j}{k}\nonumber\\
&\ge f_{k}(Pm(d))-\binom{r-2}{k+1}+ \sum_{j=0}^{r-4}\binom{r-3-j}{k}\label{eq:nonpyramidal2}\\
&= f_{k}(Pm(d))-\binom{r-2}{k+1}+\sum_{j=0}^{r-3}\binom{j}{k}\nonumber\\
&= f_{k}(Pm(d))\nonumber\\
&\ge f_k(P)\nonumber
\end{align}
where at the end we used \eqref{eq:binomial3}. Clearly none of the inqualities can be strict. Equality in \eqref{eq:nonpyramidal1} forces the conclusions that $F$ is a triplex (more precisely, $F$ is  $M(d+2-r,r-3)$), and that the vertex $v_{1}$ has degree precisely $d+1$. Equality in \eqref{eq:nonpyramidal2} forces the conclusion that  $r=3$ or $r=d$. Since each vertex $v_{i}\in S$ is nonpyramidal in $P$, and $v_{1}$ has the largest degree amongst  them, we have $\deg_{P}(v_{i})\le d+1$, for each $i$.
This settles the case $r\in [3,d]$, putting $F$ into either case (i) or case (ii).

(iii) Next we look at the case $r=2$, meaning that $F$ has $2d-1$ vertices and that $S=\left\{v_{1},v_{2}\right\}$. We want to show that $F$ falls into case (iii); let us consider the possibility that it does not. Then every vertex of $F$ is pyramidal or simple therein.

If this is the case, then \cref{lem:dplus2facets} ensures that $F$ is a $t$-fold pyramid over $\Delta(m,n)$, where $m+n+t=d-1$. Without loss of generality, assume that $m\le n$. Clearly $f_{0}(F)=(m+1)(n+1)+t=2d-1$, which implies that $mn=d-1$. This precludes the possibility that $m=1$. Hence $2\le m\le n$.

Every ridge $R$ of $P$ that is contained in $F$ is either $\Delta^{t-1}({m,n})$, $\Delta^{t}({m-1,n})$, or $\Delta^{t}({m,n-1})$ (\cref{rem:dplus2facets-facets}), and so has at least $mn+m+t$ vertices. For $i=1,2$, let $F_i$ be a facet of  $P$ containing $v_i$ but not $v_{3-i}$. We claim that each $F_i$ is a pyramid with apex $v_i$. Let $R$ be an arbitrary ridge contained in $F_i$ but not containing $v_i$. Clearly $R\subset F$ and $R\subset F_i$, which forces $R=F\cap F_i$. Thus $R$ is the unique ridge in $F_i$ not containing $v_i$, making $F_i$ a pyramid over $R$. In particular, each $v_i$ is  adjacent to every  vertex in $R$, so $v_1$ and $v_2$ between them have at least $2(mn+m+t)$ edges running into $F$. This is $(m-1)(n+1)+t$ more than the number of vertices in $F$, so $(m-1)(n+1)+t$ vertices in $F$ must be adjacent to both $v_1$ and $v_2$. Just $t$ vertices in $F$ are nonsimple in $F$, so at least $(m-1)(n+1)$ simple vertices in $F$ are nonsimple in $P$; let us fix one such vertex $u$, and note that $u$ is both nonsimple and nonpyramidal in $P$. Choose a ridge $R'$ in $F$ of the form $\Delta^{t}({m-1,n})$ that avoids $u$. Then $R'$ avoids $n+1$ vertices of $F$ altogether, and so does $F'$, the other facet corresponding to $R'$. In particular, $f_{0}(F')$ is either $2d-n$ or $2d-n-1$, depending on whether it contains one or both $v_i$. Write $f_{0}(F')$ as $d-1+d+2-r'$. Then $2d-n\le 2d-2$, implying that $r'\in [3,d+1]$.

Since $F'$  avoids the nonpyramidal nonsimple vertex $u$,  the preceding parts show that $r'$ cannot be in $[4,d-1]$, and so $f_{0}(F')$ is either $d$ (if $r'=d+1$), $d+1$ (if $r'=d$), or $2d-2$ (if $r'=3$). Recalling that $mn=d-1$, the only integer solution for $f_{0}(F')$, with the constraint
$2\le m\le n$, is $m=n=2$ and $d=5$, contradicting our assumption that $d\ge 6$. Thus $F$ falls into Part (iii) after all. From \cref{lem:dplus2facets}, this means $F$ has at least $(d-1)+3$ facets.

(iv) Finally, consider the case $r=d+1$; that is, $F$ is a simplex. We want to establish the conclusion of (iv). Assume, if possible, that every facet not containing $v_{1}$ is a simplex.

We claim that there is a ridge $R$ in $P$ that is the intersection of two facets $F'$ and $F''$, neither of which contains $v_{1}$. Otherwise, suppose that every ridge of $P$ is contained a facet which contains $v_{1}$. We now work on the dual polytope $P^{*}$ of $P$; let $\sigma$ be an anti-isomorphism from the face lattice of $P$ to the face lattice of $P^{*}$. By duality, every edge in $P^{*}$ has one vertex in the facet $\sigma({v_{1}})$ of $P^{*}$ associated with $v_{1}$. This means that there cannot be two distinct vertices outside $\sigma({v_{1}})$. As a result, $P^{*}$ would be a pyramid with base $\sigma({v_{1}})$. The dual statement of this is that  $v_{1}$ would be pyramidal at $P$, contrary to hypothesis. Thus our assumption is false, and the claims holds.

The facets $F'$ and $F''$ are both simplices, and $F'\cup F''$ contains $d+1$ vertices. Denote by $S':=\left\{v_{1}',v_{2}',\ldots,v'_{d}\right\}$  the vertices outside $F'\cup F''$; we may choose $v_{1}'=v_{1}$.   The $k$-faces in $P$ include the $k$-faces of $F'\cup F''$ and the $k$-faces containing at least one of the vertices in $S'$.  \cref{cor:number-faces}(ii)  gives the following inequalities for $k\in [2,d-2]$:
\begin{align*}
f_{k}(P)&\ge f_{k}(F'\cup F'')+\binom{d}{k}+\binom{d-1}{k}-\binom{d-2}{k}+\sum_{i=2}^{d} \binom{d-i+1}{k}\\
&\ge\binom{d}{k+1}+\binom{d-1}{k}+\binom{d}{k}+\binom{d-1}{k}-\binom{d-2}{k}+\sum_{i=2}^{d} \binom{d-i+1}{k}\\
&=\binom{d}{k+1}+\binom{d-1}{k}+\binom{d-2}{k-1}+\sum_{i=1}^{d} \binom{d-i+1}{k}\quad \text{(by \eqref{eq:binomial2})}\\
&=\binom{d}{k+1}+\binom{d-1}{k}+\binom{d-2}{k-1}+ \binom{d+1}{k+1}\quad \text{(by \eqref{eq:binomial3})}\\
&=f_{k}(Pm(d))+\binom{d-2}{k-1} \quad \text{(by \eqref{eq:pentasm})}\\
 &>f_{k}(Pm(d)) \quad \text{(since $k\in[2,d-2]$)},
\end{align*}
contrary to hypothesis. So our initial assumption in this case is false, i.e. there must be a nonsimplex facet avoiding $v_{1}$. This establishes case (iv), completing the proof of the lemma.
\end{proof}

Before continuing, we need some basic facts about {\it cyclic polytopes}. These were first defined by Gale \cite{Gal63} as follows. The {\it moment curve} in $\mathbb{R}^d$ is defined by $x(t):=(t,t^{2},\ldots,t^{d})$ for $t\in \mathbb{R}$, and the convex hull of any $n>d$ distinct points on it gives a {\it cyclic polytope} $C(n,d)$. Gale \cite{Gal63} calculated the $f$-vectors of cyclic polytopes, and showed that the combinatorial type of the polytope does not depend on which points on the moment curve have been chosen. For further details see also \cite[Sec. 4.7]{Gru03} or \cite[p. 11 ff.]{Zie95}. The cyclic 3-polytope with $n$ vertices,  $C(n,3)$ can be obtained by stacking a simplex $n-4$ times; see \cite[Fig. 1]{PerShe67}. We only need to know the following fact about cyclic 4-polyopes: $C(n,4)$ has $\h n(n-3)$  facets, and thus its dual $C(n,4)^{*}$  contains $\h n(n-3)$ vertices and $n$ facets. We note in passing that each facet of $C(n,4)^{*}$ is $C(n-1,3)^{*}$, i.e. the so called $(n-2)$-wedge; see \cite[Thm.~6]{PerShe67} and the paragraph that follows.

The next result is contained in the construction in \cite[10.4.1]{Gru03}; we make the details explicit here. The non-existence of simple 4-polytopes with six, seven or ten vertices is well known, \cite[10.4.2]{Gru03} or \cite[Lemma 2.19]{PinUgoYosEXC}.

\begin{lemma}\label{f-vector4-polytope}
For a fixed $f_0$,  the minimum value of $f_3(P)$, amongst all 4-polytopes with $f_0$ vertices, is the unique integer $n$ satisfying $\binom{n-2}{2}\le f_0\le\binom{n-1}{2}-1$. If $f_0\ge11$, then there is a simple 4-polytope with $f_0$ vertices and this minimum number of facets.
\end{lemma}

\begin{proof} For $n\ge2$, the intervals $\left[\binom{n-1}{2},\binom{n}{2}-1\right]$ are obviously disjoint, and their union is all the natural numbers. So given $f_{0}$, there is a unique integer $n$ satisfying the given inequality. First we claim that no 4-polytope with $f_0$ vertices has fewer than $n$ facets. Refer to \cite[Sec. 10.1]{Gru03} for the known inequality \[f_0\le\h f_3(f_3-3),\] which (given that $f_3$ is positive) is equivalent to $f_3\ge\h(\sqrt{8f_0+9}+3)$. If $f_0>\binom{n-2}{2}-1$, then
$$f_3>\h\left(\sqrt{(2n-5)^2}+3\right)=n-1,$$which settles the claim.

Next we prove by induction on $n\ge7$ that, if  $\binom{n-2}{2}\le f_0\le\binom{n-1}{2}-1$ and $f_0\ge11$, then there is a simple polytope with $f_0$ vertices and $n$ facets.

In the base case $n=7$  we have
$10\le f_{0}\le 14$.
This case  was settled by Br{\"u}ckner \cite{Bru08}, who completely classified the simple 4-polytopes with 7 facets. They are the simplicial 4-prism with one vertex truncated \cite[Fig. 9]{Bru08}, the polytope $\Delta({2,2})$ with one vertex truncated \cite[Fig. 10 or 10a]{Bru08}, the polytope $\Delta({1,1,2})$ (i.e. the Minkowski sum of a square and a triangle) \cite[Fig. 11]{Bru08}, the polytope $\Delta({2,2})$ with one edge truncated \cite[Fig. 12]{Bru08}, and $C(7,4)^*$ \cite[Fig. 13]{Bru08}. They have respectively 11, 12, 12, 13, and 14 vertices.

Suppose that $n\ge8$. Then
\begin{align*}
	\binom{n-3}{2}< \binom{n-2}{2}-3<  \binom{n-2}{2}-2< \binom{n-2}{2}-1.
\end{align*}
Therefore, by the induction hypothesis, there are two simple $4$-polytopes with $n-1$ facets, and $f_{0}=\binom{n-2}{2}-3$ and $f_{0}=\binom{n-2}{2}-2$ vertices, respectively. Truncating a vertex from these and applying \cref{lem:examples-all-neighbours-simple}(i) then establishes the conclusion for  $f_0=\binom{n-2}{2}$ and $\binom{n-2}{2}+1$, respectively.

Because $C(n-1,4)^*$ contains $\h (n-1)(n-4)=\binom{n-2}{2}-1$ vertices, $n-1$ facets,  and a 2-face $F$ with $n-3$ vertices, truncating $m$ vertices of  $F$, for $1\le m\le n-4$, as in \cref{lem:examples-all-neighbours-simple}(ii), will give us a simple polytope with $n$ facets and $\binom{n-2}{2}-1+m+2$ vertices. This establishes the conclusion for
$$\binom{n-2}{2}+2\le f_0\le\binom{n-2}{2}-1+n-2=\binom{n-1}{2}-1,$$
thereby completing the proof of the lemma.
\end{proof}

 We require another concept, related to the definition of beyond, and a result from \cite[Sec. 5.2]{Gru03}. A point $v$ is {\it beneath} a facet of a polytope if $v$ belongs to the open halfspace that is determined by the supporting hyperplane for that facet and contains the interior of the polytope.

\begin{proposition}
[{\cite[Thm.~5.2.1]{Gru03}}]
\label{prop:beneath-beyond} Let $P$ and $P'$ be two $d$-polytopes in $\mathbb{R}^d$, and let $v$ be a vertex of $P'$ such that $v\not\in P$ and $P'=\conv (P\cup \{v\})$. Then a face $F$ of $P$ is a face of $P'$ if and only if there exists a facet $J$ of $P$ such that  $F\subseteq J$ and $v$ is beneath  $J$.
\end{proposition}

We show that  the $d$-pentasm is the unique minimiser of the number of $k$-faces among the $d$-polytopes with $2d+1$ vertices and at least one vertex that is both nonpyramidal and nonsimple.

\begin{theorem}\label{thm:main} Let $d\ge 5$. Let $P$ be a $d$-polytope with $2d+1$ vertices, at least one of which is both nonpyramidal and nonsimple. Suppose $f_k(P)\le f_k(Pm(d))$ for some $k\in [1,d-2]$. Then $P$ is a pentasm.
\end{theorem}

\begin{proof} By induction on $d$. The base case $d=5$ is contained in  \cref{lem:five-dim}, so we assume $d\ge6$.

The result is already known if $k=d-2$ (\cref{prop:ridges}(ii)). The case $k=1$ for $d\ge 5$ is equivalent to the statement that a $d$-polytope with $2d+1$ vertices and at most $d^{2}+d-1$ edges  is  a $d$-pentasm, i.e. to \cref{prop:edges}.  So we assume $2\le k\le d-3$.

Amongst the facets in $P$ that avoid at least one nonpyramidal, nonsimple vertex, choose $F$ with  a maximum number of vertices. By \cref{lem:nonpyramidal}, $F$ has exactly $d-1+d+2-r$ vertices with $r=2,3$ or $d$.   There are $r$ vertices outside $F$.

We begin with the case $r=2$, namely \cref{lem:nonpyramidal}(iii).
Then $F$  has a nonpyramidal, nonsimple vertex and  $2(d-1)+1$ vertices. Clearly the two vertices $v_1$ and $v_2$ outside $F$ are nonpyramidal. They must be adjacent, and one of them must be adjacent to at least $d$ vertices in $F$. Without loss of generality, suppose that $v_1$ is nonsimple and  its degree in $P$ is not smaller than that of $v_2$. From \cref{prop:number-faces}, we get a short sequence $F_{1},F_{2}$ of faces such that $F_1=P$, $F_2$ is a facet, $v_{i}\in F_{i}$, and $v_1\not\in F_2$. Any $k$-face of $P$ falls into one of four disjoint groups: either it
\begin{itemize}
    \item is contained in $F$, or
    \item contains $v_1$, or
    \item contains $v_2$ and is contained in $F_2$, or
    \item contains $v_2$ but not $v_1$, and is not contained in $F_2$.
\end{itemize}
Denote by $n$ the number of $k$-faces in this last group. \cref{prop:number-faces}, together with our hypothesis, now  yields the following estimates:
\begin{align}
\label{eq:main-1}
    f_k(P)&= f_k(F)+\phi_{k-1}(\deg_P(v_1),{d-1})+\binom{d-1}{k}+n\\
    &\ge f_k(Pm(d-1))+\binom{d}{k}+\binom{d-1}{k}-\binom{d-2}{k}+\binom{d-1}{k}+0\nonumber\\
    &= \binom{d}{k+1}+\binom{d-1}{k+1}+\binom{d-2}{k}+\binom{d}{k}+\binom{d-1}{k}-\binom{d-2}{k}+\binom{d-1}{k}\nonumber\\
    &= \binom{d+1}{k+1}+\binom{d}{k+1}+\binom{d-1}{k}\nonumber\\
    &=f_k(Pm(d))\nonumber\\
    &\ge f_k(P).\nonumber
\end{align}
Clearly none of these inequalities can be strict, which leads to the following conclusions.
\begin{enumerate}
    \item [(a)]  $f_k(F)= f_k(Pm(d-1))$.
    \item[(b)]  The vertex $v_1$ has  degree precisely $d + 1$, and therefore $\deg_P(v_2)\le d+1$.
    \item[(c)] The vertex $v_2$  is simple in $F_2$.
    \item[(d)] $n=0$.
\end{enumerate}

From Observation (a) the induction hypothesis kicks in, to inform us that $F$ is a pentasm. Observation (d) means that every $k$-face containing $v_2$ but not $v_1$ must be contained in $F_2$. In other words, $v_1$ is the only neighbour of $v_2$ not contained in $F_2$. Combined with Observation (c), we obtain that $v_2$ is simple in $P$.

Since $v_1$ has degree $d+1$ and $v_2$ is simple in $P$, the number of edges of $P$ is $f_1(F)+2d$, the number of edges of $Pm(d)$. Again  \cref{prop:edges}  yields that $P$ is a pentasm.

Now consider the cases $r=3, d$. These are respectively cases (ii) and (i) of \cref{lem:nonpyramidal}, so we may assume that every facet avoiding a nonpyramidal nonsimple vertex has at most $2d-2$ vertices and that every vertex in $P$ is either pyramidal or has degree at most $d+1$. We will see that both cases contradict our hypothesis.

If $r=3$, then some facet $F$ avoiding a nonpyramidal, nonsimple vertex has exactly $2d-2$ vertices. Such a facet must be a simplicial $d$-prism (\cref{lem:nonpyramidal}(ii)), and so no vertex in $P$ is pyramidal, whence every vertex of $P$ is either simple or has   degree $d+1$.
If there were only  $d-2$ nonsimple vertices, $P$ would have only $d^2+d-1$ edges, and thus would be a $d$-pentasm by \cref{prop:edges}. Therefore, we may assume $P$ has at least $ d-1$ nonsimple vertices. For $d\ge7$, this means $P$ contains at least six nonsimple vertices. It is easy to see that $2f_1(P)$ is the sum of the degrees of the vertices; if $d=6$, this implies that the number of  nonsimple vertices is even, and again $P$ contains at least six nonsimple vertices.

We label the vertices of $F$ as $u_1, u_2,\ldots, u_{d-1},x_1, x_2,\ldots, x_{d-1}$ so that  the sets $u_1, u_2,\ldots, u_{d-1}$ and $x_1, x_2,\ldots, x_{d-1}$ respectively form two $(d-2)$-simplices, and  $u_{i}x_{i}$ is an edge $E_{i}$ for $i=1,\ldots,d-1$.

With only three vertices outside, some vertex in $F$, without loss of generality  say $u_1$, must be nonsimple. Denote by $R$ the $(d-2)$-face of $F$ containing $u_2,\ldots, u_{d-1}, x_2,\ldots, x_{d-1}$, and by $F'$ the other facet of $P$ containing $R$. Since  $F'$ has at least $2d-3$ vertices, which is greater than $d+1$ for $d\ge 6$, but does not contain the nonsimple vertex $u_1$, $F'$ must also be a simplicial $(d-1)$-prism; let $u_d,x_d$ be the two vertices in $F'\setminus F$.  As remarked in the introduction, $F$ being a simplicial prism ensures that the lines containing $E_1,E_2,\ldots, E_{d-1}$ are all either parallel or concurrent at a point outside $F$.  Similarly, as $F'$ is a simplicial prism, the lines containing $E_2,\ldots E_{d-1},E_{d}$ must all be either parallel or concurrent at a point outside $F'$. Consequently, the same conclusion holds for $E_1,E_2,\ldots, E_{d}$. Denote by $w$ the unique vertex not in $F\cup F'$.

Since $P$ has  at least six nonsimple vertices, there must be a nonsimple vertex different from $u_1,u_{d},x_1,x_{d},w$. Without loss generality, this vertex is either $u_2$ or $x_2$. Let $R'$ denote the ridge containing the edges $E_1,E_3,\ldots, E_{d-1}$. Then $R'\subset F$; let $F''$ be the other facet containing $R'$. Then $F''$ avoids a nonsimple vertex ($u_2$ or $x_2$) and so, as before, must be a simplicial prism. In particular $F''$ contains precisely two of $u_d,x_d,w$. Suppose $F''$ contains $u_d$ and $w$. For $F''$ to be a simplicial prism, $E=u_dw$ must be an edge of it, and the lines containing $E_1,E_3,\ldots E_{d-1},E$ must all be either parallel or  concurrent at a point outside $F''$. But then all the lines containing $E_1,E_2,E_3,\ldots, E_{d},E$ must be parallel or concurrent at a point outside $F''$. Since $E$
and $E_d$ are concurrent at $u_d$, this is impossible. Similarly, the two vertices in $F''\setminus R'$ cannot be $x_d$ and $w$; they must be $u_d$ and $x_d$. Since $F''$ is a simplicial prism, this implies that $u_1u_d$ and $x_1x_d$ are edges of $F''$, and of $P$.

Thus the convex hull of $F\cup F'$ is a simplicial $d$-prism, which we will denote by $Q$. Then $P=\conv (Q\cup \{w\})$. We have that $F$ and $F'$ are facets of $Q$ and $P$. In this setting, \cref{prop:beneath-beyond} ensures that $w$ is beneath $F$ and $F'$. Besides, because $w$ has degree at most $d+1$ in $P$, it cannot be beyond  any of the simplicial prism facets of $Q$. Thus $w$ is beyond only one of the simplex  facets of $Q$, which implies that $P$ is a capped prism (\cref{subsection:capped}). But any capped prism has more $k$-faces than the pentasm, for $k\le d-2$, according to \cref{lem:capped}.

Now suppose that $r=d$, i.e. the facet $F$ is a $(d-3)$-fold pyramid over a quadrilateral. Equivalently,  $F$ is a pyramid with apex $u_{i}$ and base $R_{i}$ for each $i\in[1,d-3]$ where $R_{i}$ is a $(d-4)$-fold pyramid over a quadrilateral. Let $v_{1}$ be a nonsimple vertex outside $F$, and let $S$ denote the set of vertices outside $F$.

Suppose that some vertex $u_{i}$ of $F$ is nonpyramidal in $P$, say $u_{1}$. Then the other facet $F'$ containing $R_{1}$ must be a $(d-3)$-fold pyramid over a quadrilateral as well, because $u_{1}\not\in F'$, $u_{1}$ is nonpyramidal and nonsimple in $P$, and $F$ has the largest number of vertices among all facets avoiding a nonpyramidal, nonsimple vertex in $P$. Assume that $v_{1}\in F'$; we may do so without loss of generality because the vertex in $S\cap V(F')$ is nonpyramidal and nonsimple in $P$. \cref{cor:number-faces}(ii) applied to $S\setminus \left\{v_{1}\right\}$ leads to a strict inequality for $k\in [2,d-2]$:
\begin{align*}
f_{k}(P)&\ge f_{k}(F\cup F')+\sum_{i=1}^{d-1}\binom{d-i+1}{k}\\
&\ge \binom{d}{k+1}+\binom{d-2}{k}+\binom{d-1}{k}+\binom{d-3}{k-1}+\sum_{i=1}^{d-1}\binom{d-i+1}{k}\\
&=\binom{d}{k+1}+\binom{d-2}{k}+\binom{d-1}{k}+\binom{d-3}{k-1}+\binom{d+1}{k+1}\quad \text{(by \eqref{eq:binomial3})}\\
&=f_{k}(Pm(d))+\binom{d-2}{k}+\binom{d-3}{k-1}\\
&>f_{k}(Pm(d)).
\end{align*}

Finally consider the case that every $u_{i}$ is pyramidal in $P$: then $P$ is a $(d-3)$-fold pyramid over a $3$-polytope with $d+4$ vertices. It is more convenient now to consider $P$ as a $(d-4)$-fold pyramid over a 4-polytope $Q$ with $d+5$ vertices, equivalently as the free join of a $(d-5)$-simplex and $Q$.

\cref{f-vector4-polytope} gives us a simple 4-polytope $Q'$ with $f_0(Q')=f_0(Q)=d+5\ge11$ and $f_3(Q')$ minimal amongst all 4-polytopes with $d+5$ vertices. Being simple, we have $f_1(Q')=2f_0(Q')=2f_0(Q)\le f_1(Q)$. Euler's relation then implies that $f_2(Q')\le f_2(Q)$ as well. Denote by $P'$ the free join of a $(d-5)$-simplex $T$ and $Q'$. Then, by \cref{lem:join-faces} we again have
\begin{align*}
f_{k}(P)&= f_{k}(T)+f_{k-1}(T)f_0(Q)+f_{k-2}(T)f_1(Q)+f_{k-3}(T)f_2(Q)+f_{k-4}(T)f_3(Q)\\
&{}\quad+f_{k-5}(T)\\
&\ge f_{k}(T)+f_{k-1}(T)f_0(Q')+f_{k-2}(T)f_1(Q')+f_{k-3}(T)f_2(Q')+f_{k-4}(T)f_3(Q')\\
&{}\quad+f_{k-5}(T)\\
&=f_{k}(P')\\
&>f_{k}(Pm(d)),
\end{align*}
where the last inequality $f_{k}(P')>f_{k}(Pm(d))$ comes from \cref{lem:pyramids-dplus3-facets}.
This concludes this case, and with it, the proof of the theorem.
\end{proof}

An immediate consequence of \cref{lem:pyramids-dplus3-facets} and \cref{thm:main}  is the high-dimensional case of our main result.

\begin{theorem}
\label{thm:main-II}
Let $d\ge 6$. Let $P$ be a $d$-polytope with $2d+1$ vertices such that $f_k(P)\le f_k(Pm(d))$ for some $k\in [1,d-2]$.
\begin{enumerate}
\item  If at least one vertex in $P$ is both nonpyramidal and nonsimple, then $P$ is a pentasm.
\item If every vertex in $P$ is pyramidal or simple, then $P$ is a $t$-fold pyramid over $\Delta(r,s)$ for some $r,s>0$ and $t\ge 0$ such that $d=r+s+t$.
\end{enumerate}
\end{theorem}

Combining  \cref{thm:main-II} with \cref{cor:dprime} and the low-dimensional cases in \cref{lem:five-dim}, we get the conclusion of  \cref{thm:2dplus1-bound}.

When $d$ is prime, the pentasm is the unique minimiser of $f_k$ for all $k\le d-2$. But
when $d$ is composite, \cref{thm:2dplus1-bound} gives us a finite list of candidates for the minimiser of each $f_k$. We conjecture the following more precise conclusion, which we have verified for all $d\le100$.
 \begin{enumerate}   \item If  $d$ is composite, and $k\le d/2$,  then the $d$-pentasm is the unique minimiser for $f_k$.
    \item If  $d=rs$ is composite,  $r$ is the smallest prime factor of $d$, and $k> d/2$,  then a $(d-r-s)$-fold pyramid over $\Delta({r,s})$ is the unique minimiser of $f_k(P)$.
\end{enumerate}

\bibliographystyle{amsplain}

\providecommand{\bysame}{\leavevmode\hbox to3em{\hrulefill}\thinspace}
\providecommand{\MR}{\relax\ifhmode\unskip\space\fi MR }
\providecommand{\MRhref}[2]{%
  \href{http://www.ams.org/mathscinet-getitem?mr=#1}{#2}
}
\providecommand{\href}[2]{#2}

\end{document}